\documentclass[3p]{elsarticle}

\usepackage{amsmath,amsthm,amssymb,mathrsfs}
\usepackage{enumerate}
\usepackage{slashed}
\usepackage{txfonts,dsfont}
\usepackage{aliascnt}
\usepackage[breaklinks,colorlinks]{hyperref}

\makeatletter
\newcommand{\autorefcheckize}[1]{%
  \expandafter\let\csname @@\string#1\endcsname#1%
  \expandafter\DeclareRobustCommand\csname relax\string#1\endcsname[1]{%
    \csname @@\string#1\endcsname{##1}\wrtusdrf{##1}}%
  \expandafter\let\expandafter#1\csname relax\string#1\endcsname
}
\makeatother

\theoremstyle{plain}

\newtheorem{theorem}{Theorem}[section]

\newaliascnt{lem}{theorem}
\newtheorem{lem}[lem]{Lemma}
 \aliascntresetthe{lem}

\newaliascnt{cor}{theorem}

 \aliascntresetthe{cor}

\newaliascnt{prop}{theorem}
\newtheorem{prop}[prop]{Proposition}
 \aliascntresetthe{prop}

\theoremstyle{remark}

\newtheorem{rem}{Remark}[section]

\newtheorem*{claim}{Claim}

\theoremstyle{definition}
\newtheorem{defn}{Definition}[section]

\newtheorem{eg}{Example}[section]

\numberwithin{equation}{section}

\newcommand{\abs}[1]{\left\lvert#1\right\rvert}
\newcommand{\set}[1]{\left\{#1\right\}}
\newcommand{\hin}[2]{\left\langle#1,#2\right\rangle}

\newcommand*{\To}{\longrightarrow}
\newcommand*{\Rmn}[1]{\uppercase\expandafter{\romannueral#1}}
\newcommand*{\dif}{\mathop{}\!\mathrm{d}}

\DeclareMathOperator{\Div}{div}
\DeclareMathOperator{\trace}{tr}

\journal{}

\allowdisplaybreaks

\begin{document}

\begin{frontmatter}

\title{Complete Willmore Legendrian surfaces in $\mathbb{S}^5$ are  minimal Legendrian surfaces \tnoteref{ls}}

\author[whu1,whu2]{Yong Luo}
\ead{yongluo@whu.edu.cn}

\author[whu1,whu2,mpg]{Linlin Sun\corref{sll1}}
\ead{sunll@whu.edu.cn}

\tnotetext[ls]{This work was partially supported by the NSF of China (Nos. 11501421, 11801420, 11971358) and the Youth Talent Training Program of Wuhan University.}

\address[whu1]{School of Mathematics and Statistics, Wuhan University, Wuhan 430072, China}
\address[whu2]{Hubei Key Laboratory of Computational Science, Wuhan University, Wuhan, 430072, China}
\address[mpg]{Max Planck Institute for Mathematics in the Sciences, Inselstrasse 22, 04103 Leipzig, Germany}

\cortext[sll1]{Corresponding author.}

\begin{abstract}
In this paper we continue to consider Willmore Legendrian surfaces and csL Willmroe surfaces in $\mathbb{S}^5$, notions introduced by Luo in \cite{Luo}. We will prove that every complete Willmore Legendrian surface in $\mathbb{S}^5$ is minimal and find nontrivial examples of csL Willmore surfaces in $\mathbb{S}^5$.

\end{abstract}


\end{frontmatter}


\section{Introduction}

Let $\Sigma$ be a Riemann surface, $(M^n,g)=\mathbb{S}^n$ or $\mathbb{R}^n(n\geq3)$ the unit sphere or the Euclidean space with standard metrics and $f$ an immersion from $\Sigma$ to $M$. Let $B$ be the second fundamental form of $f$ with respect to the induced metric, $H$ the mean curvature vector field of $f$ defined by
$$H=\trace B,$$
$\kappa_M$  the Gauss curvature of $df(T\Sigma)$ with respect to the ambient metric $g$ and $\dif\mu_f$  the area element on $f(\Sigma)$. The Willmore functional of the immersion $f$ is then defined by
\begin{align*}
W(f)=\int_{\Sigma}\left(\dfrac{1}{4}\abs{H}^2+\kappa_M\right)\dif\mu_f,
\end{align*}

For a smooth and compactly supported variation $f: \Sigma\times I\mapsto M$ with $\phi=\partial_tf$ we have the following first variational formula (cf.~\cite{Tho, Wei})
\begin{eqnarray*}
\dfrac{d}{dt}W(f)=\int_\Sigma\hin{ \overrightarrow{W}(f)}{\phi} \dif\mu_f,
\end{eqnarray*}
with $\overrightarrow{W}(f)=\sum_{\alpha=3}^n\overrightarrow{W}(f)^\alpha e_\alpha$, where $\{e_\alpha:3\leq\alpha\leq n\}$ is a local orthonormal frame of the normal bundle of $f(\Sigma)$ in $M$ and
\begin{eqnarray*}
\overrightarrow{W}(f)^\alpha=\dfrac12\left(\Delta H^\alpha+\sum_{i,j,\beta}h_{ij}^\alpha h_{ij}^\beta H^\beta-2\abs{H}^2H^\alpha\right)
,\quad 3\leq\alpha\leq n,
\end{eqnarray*}
where $h_{ij}^\alpha$ is the component of $B$ and $H^\alpha$ is the trace of $\left(h_{ij}^\alpha\right)$.

A smooth immersion $f:\Sigma\mapsto  M$ is called a {\it Willmore immersion}, if it is a critical point of the Willmore functional $W$. In other words, $f$ is a Willmore immersion if and only if it satisfies
\begin{eqnarray}\label{ELE}
\Delta H^\alpha+\sum_{i,j,\beta}h_{ij}^\alpha h_{ij}^\beta H^\beta-2\abs{H}^2H^\alpha=0, \quad 3\leq\alpha\leq n.
\end{eqnarray}

When $(M,g)=\mathbb{R}^3$, Willmore \cite{Wil} proved that the Willmore energy of closed surfaces are larger than or equal to $4\pi$ and equality holds only for round spheres. When $\Sigma$ is a torus, Willmore conjectured that the minimum is $2\pi^2$ and it is attained only by the Clifford torus, up to a conformal transformation of $\mathbb{R}^3$ \cite{Whi,Chen}, which was verified by Marques and Neves in \cite{MN}. When $(M,g)=\mathbb{R}^n$, Simon \cite{Sim}, combined with the work of Bauer and Kuwert \cite{BaK}, proved the existence of an embedded surface which minimizes the Willmore functional among closed surfaces of prescribed genus. Motivated by these mentioned papers, Minicozzi \cite{Mi} proved the existence of an embedded torus which minimizes the Willmore functional in a smaller class of Lagrangian tori in $\mathbb{R}^4$. In the same paper Minicozzi conjectured that the Clifford torus minimizes the Willmore functional in its Hamiltonian isotropic class, which he verified has a close relationship with Oh's conjecture \cite{Oh90,Oh93}. We should also mention that before Minicozzi, Castro and Urbano proved that the Whitney sphere in $\mathbb{R}^4$ is the only minimizer for the Willmore functional among closed Lagrangian sphere. This result was further generalized by Castro and Urbano in \cite{CU93} where they proved that the Whitney sphere is the only closed Willmore Lagrangian sphere (a Lagrangian sphere which is also a Willmore surface) in $\mathbb{R}^4$. Examples of Willmore Lagrangian tori (Lagrangian tori which also are Willmroe surfaces) in $\mathbb{R}^4$ were constructed by Pinkall \cite{Pin} and Castro and Urbano \cite{CU}. Motivatied by these works, Luo and Wang \cite{LW} considered the variation of the Willmore functional among Lagrangian surfaces in $\mathbb{R}^4$ or variation of a Lagrangian surface of the Willmore functional among its Hamiltonian isotropic class in $\mathbb{R}^4$, whose critical points are called LW or HW surfaces respectively. We should also mention that Willmroe type functional of Lagrangian surfaces in $\mathbb{CP}^2$ were studied by Montiel and Urbano \cite{MU} and Ma, Mironov and Zuo \cite{MMZ}.

Inspired by the study of the Willmore functional for Lagrangian surfaces in $\mathbb{R}^4$, Luo \cite{Luo} naturally considered the Willmore functional of Legendrian surfaces in $\mathbb{S}^5$. 

\begin{defn}
A Willmore and Legendrian surface in $\mathbb{S}^5$ is called \textbf{a Willmore Legendrian surface}.
\end{defn}

\begin{defn} A Legendrian surface in $\mathbb{S}^5$ is called \textbf{a contact stationary Legendrian Willmore surface} (in short, a csL Willmore surface) if it is a critical point of the Willmore functional under contact deformations.
\end{defn}
Luo \cite{Luo} proved that Willmore Legendrian surfaces in $\mathbb{S}^5$ are csL surfaces (see \autoref{CSL}). In this paper, we continue to study Willmore Legendrian surfaces and csL Willmore surfaces in $\mathbb{S}^5$. Surprisingly we will prove that every complete Willmore Legendrian surface in $\mathbb{S}^5$ must be a minimal surface (\autoref{main thm1}).  We also find nontrivial examples of csL Willmore surfaces from csL surfaces in $\mathbb{S}^5$ for the first time, by exploring relationships between them (\autoref{pro1}).

The method here we used  to find nontrivial csL Willmore surfaces in $\mathbb{S}^5$ in Section 4 should also be useful in discovering nontrivial HW surfaces in $\mathbb{R}^4$ introduced by Luo and Wang in \cite{LW}. We will consider this problem in the future.

\section{Willmore Legendrian surfaces in \texorpdfstring{$\mathbb{S}^5$}{the unit sphere}}

In this section we will prove that every complete Willmore Legendrian surfaces in $\mathbb{S}^5$ is minimal. Firstly we briefly record several facts about Legendrian surfaces in $\mathbb{S}^5$. We refer the reader to consult \cite{Bl} for more materials about the contact geometry.

Let $\mathbb{S}^5$, the 5-dimensional unit sphere, be the standard Sasakian Einstein manifold with contact one form $\alpha$, almost complex structure $J$, Reed field $\textbf{R}$ and canonical metric $g$. Let $\Sigma$ be a closed surface of $\mathbb{S}^{5}\subset\mathbb{C}^{3}$. We say that $\Sigma$ is Legendrian if
\begin{align*}
    JT\Sigma\subset T^\nu\Sigma,\quad JF\in\Gamma\left(T^\nu\Sigma\right)
\end{align*}
where $F: \Sigma\To\mathbb{S}^{5}$ is the position vector and $T\Sigma, T^\nu\Sigma$ are tangent and normal bundles of $\Sigma$ respectively. We say that $\Sigma$ is a minimal Legendrian surface of $\mathbb{S}^{5}$ if $\Sigma$ is a minimal and Legendrian surface of $\mathbb{S}^{5}$. Define
\begin{align*}
    \sigma\left(X,Y,Z\right)\coloneqq\hin{\mathbf{B}\left(X,Y\right)}{JZ},\quad\forall X,Y,Z\in T\Sigma.
\end{align*}
The Weingarten equation implies that
\begin{align*}
    \sigma\left(X,Y,Z\right)=\sigma\left(Y,X,Z\right).
\end{align*}
Moreover, by definition, one can check that $\sigma$ is a three order symmetric tensor, i.e.,
\begin{align}\label{tri-symm.}
    \sigma\left(X,Y,Z\right)=\sigma\left(Y,X,Z\right)=\sigma\left(X,Z,Y\right).
\end{align}
The Gauss equation, Codazzi equation and Ricci equation becomes
\begin{align}
    R\left(X,Y,Z,W\right)=&\hin{X}{Z}\hin{Y}{W}-\hin{X}{W}\hin{Y}{Z} \nonumber\\
    &+\sigma\left(X,Z,e_i\right)\sigma\left(Y,W,e_i\right)-\sigma\left(X,W,e_i\right)\sigma\left(Y,Z,e_i\right), \label{gauss equ.}\\
    \left(\nabla_{X}\sigma\right)\left(Y,Z,W\right)=&\left(\nabla_{Y}\sigma\right)\left(X,Z,W\right),\nonumber\\
    R^{\bot}\left(X,Y,JZ,JW\right)=&R\left(X,Y,Z,W\right),\nonumber
\end{align}
where $\set{e_i}$ is an orthonormal basis of $T\Sigma$. The Codazzi equation implies
\begin{align}\label{four-symm.}
    \left(\nabla_{X}\sigma\right)\left(Y,Z,W\right)=\left(\nabla_{Y}\sigma\right)\left(X,Z,W\right)=\left(\nabla_{X}\sigma\right)\left(Z,Y,W\right)=\left(\nabla_{X}\sigma\right)\left(Y,W,Z\right),
\end{align}
i.e., $\nabla\sigma$ is a fourth order symmetric tensor.

 Recall that
 \begin{defn}\label{CSL}
 $\Sigma$ is a \emph{csL surface} in $\mathbb{S}^5$ if it is a critical point of the volume functional among Legendrian surfaces.
 \end{defn}
 CsL surfaces in $\mathbb{S}^5$ satisfy the following Euler-Lagrange equation (\cite{CLU,Ir}):
\begin{align*}
\Div(JH)=0.
\end{align*}
It is obvious that $\Sigma$ is csL in $\mathbb{S}^5$ when $\Sigma$ is minimal. The following observation  is very important for the study of csL surfaces.
\begin{lem}
 $\Sigma$ is csL  in $\mathbb{S}^5$  iff  $JH$ is a harmonic vector field.
\end{lem}
By using the Bochner formula for harmonic vector fields (cf. \cite{Jo}), we get
\begin{lem}\label{lem:bochner}
If $\Sigma$ is csL in $\mathbb{S}^5$ , then
 \begin{align*}
\dfrac12\Delta\abs{H}^2=&|\nabla(JH)|^2+Ric(JH,JH).
\end{align*}
\end{lem}

From \autoref{lem:bochner} it is easy to see that we have
\begin{lem}\label{lem:Gauss}
If $\Sigma\subset\mathbb{S}^5$ is csL and non-minimal, then the zero set of $H$ is isolate and
\begin{align*}
\Delta\log\abs{H}=\kappa
\end{align*}
provided $H\neq0$, where $\kappa$ is the Gauss curvature of $\Sigma$.
\end{lem} 

We then prove that every complete Willmore Legendrian surface in $\mathbb{S}^5$ must be a minimal surface. Firstly, we rewrite the Willmore operator acting on Legendrian surfaces, i.e., we prove the following 
\begin{prop}\label{pro0}
Assume that $\Sigma$ is a Legendrian surface in $\mathbb{S}^5$, then its Willmore operator can be written as
\begin{eqnarray*}
\vec{W}(\Sigma)=\dfrac12\set{-J\nabla\Div(JH)+B(JH,JH)-\dfrac12\abs{H}^2H-2\Div(JH)\textbf{R}}.
\end{eqnarray*}
In particular, the Euler-Lagrange equation of Willmore Legendrian surfaces in $\mathbb{S}^5$ is
\begin{eqnarray}\label{equ4}
-J\nabla\Div(JH)+B(JH,JH)-\dfrac12\abs{H}^2H-2\Div(JH)\textbf{R}=0.
\end{eqnarray}
\end{prop}

\begin{proof}
Let $\{\nu_1,\nu_2,\textbf{R}\}$ be a local orthonormal frames of the normal bundle of $\Sigma$, then the Willmore equation \eqref{ELE} can be rewritten as
\begin{eqnarray*}
\Delta^{\nu}H+\sum_\alpha\langle A^\alpha, A^H\rangle\nu_{\alpha}-\dfrac12\abs{H}^2H=0.
\end{eqnarray*}
Note that by (2.8) in \cite{Luo} we have 
\begin{align*}
\nabla^\nu_X(JY)&=(\bar{\nabla}_X(JY))^\nu
\\=&((\bar{\nabla}_XJ)Y+J\bar{\nabla}_XY)^\nu
\\&=J\nabla_XY+g(X,Y)\textbf{R}
\end{align*}
for $X,Y\in\Gamma\left(T\Sigma\right)$, where $\bar{\nabla}$ denotes the covariant derivative of $\mathbb{S}^5$. Choose a local orthonormal frame field around $p$ with $\nabla_{e_i}e_j\vert_p=0$, then
\begin{eqnarray*}
&&J\nabla_{e_i}(JH)
\\&=&\nabla^\nu_{e_i}(J(JH))-g(e_i,JH)\textbf{R}
\\&=&-\nabla^\nu_{e_i}H-g(e_i,JH)\textbf{R}
\end{eqnarray*}
and
\begin{eqnarray*}
J\nabla_{e_i}(\nabla_{e_i}(JH))&=&\nabla^\nu_{e_i}(J\nabla_{e_i}(JH))-g(e_i,\nabla_{e_i}JH)\textbf{R}
\\&=&\nabla^\nu_{e_i}(-\nabla^\nu_{e_i}H-g(e_i,JH)\textbf{R})-g(e_i,\nabla_{e_i}JH)\textbf{R}
\\&=&-\nabla^\nu_{e_i}\nabla^\nu_{e_i}H-2g(e_i,\nabla_{e_i}(JH)\textbf{R}-g(e_i,JH)\left(\bar{\nabla}_{e_i}\textbf{R}\right)^{\nu}
\\&=&-\nabla^\nu_{e_i}\nabla^\nu_{e_i}H-2g(e_i,\nabla_{e_i}(JH)\textbf{R}-g(H,Je_i)Je_i,
\end{eqnarray*}
where in the last equality we used (2.7) in \cite{Luo}. Therefore we obtain
\begin{eqnarray*}
\Delta^{\nu}H=-J\Delta(JH)-H-2\Div(JH)\textbf{R},
\end{eqnarray*}
which implies that $\Sigma$ satisfies the following equation
\begin{eqnarray*}
-J\Delta(JH)+\sum_\alpha\hin{A^{\alpha}}{A^H}\nu_{\alpha}-\dfrac12\left(2+\abs{H}^2\right)H-2\Div (JH)\textbf{R}=0.
\end{eqnarray*}
In addition, by \cite[Lemma 2.9]{Luo}, the dual one form of $JH$ is closed, thus by the Ricci identity we have
\begin{eqnarray*}
\Delta(JH)=\nabla\Div(JH)+\kappa JH.
\end{eqnarray*}
The Proposition is then a consequence of the following Claim together with above two identities.
\begin{claim}
\begin{align*}
2\kappa=2+\abs{H}^2-|B|^2,\\
\sum_\alpha\hin{A^\alpha}{A^H}\nu_{\alpha}-\dfrac12|B|^2H=B(JH,JH)-\dfrac12\abs{H}^2H.
\end{align*}
\end{claim}
\proof The first equation is obvious by the Gauss equation \eqref{gauss equ.}. The second equation can be proved by the Gauss equation \eqref{gauss equ.} and the tri-symmetry of the tensor $\sigma$ (see \eqref{tri-symm.}). To be precise, for every tangent vector field $Z\in T\Sigma$ we have
\begin{eqnarray*}
&&\langle B(JH,JH),JZ\rangle-\sum_\alpha\langle A^\alpha,A^H\rangle\langle\nu_{\alpha},JZ\rangle
\\&=&-\langle B(Z,JH),H\rangle-\sum_{i,j}\langle B(e_i,e_j),JZ\rangle\langle B(e_i,e_j),H\rangle
\\&=&\sum_{i,j}\langle B(Z,e_j),Je_i\rangle\langle B(JH,e_j),e_i\rangle-\langle B(Z,JH),H\rangle
\\&=&\sum_j\langle B(Z,e_j),B(JH,e_j)\rangle-\langle B(Z,JH),H\rangle
\\&=&Ric(Z,JH)-\langle Z,JH\rangle
\\&=&\left(\kappa-1\right)\langle Z,JH\rangle
\\&=&\dfrac12\left(\abs{H}^2-\abs{B}^2\right)\langle Z,JH\rangle
\\&=&\dfrac12\left(\abs{B}^2-\abs{H}^2\right)\langle H,JZ\rangle.
\end{eqnarray*}
This completes the proof of the second equation.

\end{proof}

Now we are in position to prove the following
\begin{theorem}\label{main thm1}
Every complete Willmore Legendrian surface in $\mathbb{S}^5$ is a minimal surface.
\end{theorem}
\begin{proof}
 We prove by a contradiction argument. Assume that $\Sigma$ is a complete Willmore Legendrian surface in $\mathbb{S}^5$ which is not a minimal surface. If $H\neq0$, then let $\set{e_1=\frac{JH}{\abs{H}},e_2}$ be a local orthonormal frame field of $T\Sigma$. From \eqref{equ4} we have
 \begin{align*}
 B(e_1,e_1)=-\frac{1}{2}\abs{H}Je_1,
 \end{align*}
 which also implies that
 \begin{align*}
 B(e_2,e_2)=-\dfrac{1}{2}\abs{H}Je_1,\quad h_{11}^2=0.
 \end{align*}
Then by the Gauss equation \eqref{gauss equ.} we have
\begin{eqnarray*}
 \kappa&=&1+\hin{B(e_1,e_1)}{B(e_2,e_2)}-\abs{B(e_1,e_2)}^2
 \\&=&1+\dfrac14\abs{H}^2-\abs{h_{12}^1}^2-\abs{h_{12}^2}^2
 \\&=&1+\dfrac14\abs{H}^2-\abs{h_{22}^1}^2
 \\&=&1.
 \end{eqnarray*}
 Since $\Sigma$ is a Willmore Legendrian surface, from \eqref{equ4} we see that $\Div(JH)=0$. By \autoref{lem:Gauss}  the minimal points of $\Sigma$ are discrete and so the Gauss curvature of $\Sigma$ equals one everywhere on $\Sigma$, therefore $\Sigma$ is compact by Bonnet-Myers theorem. Apply \autoref{lem:bochner} to obtain  that on $\Sigma$
 \begin{eqnarray*}
 \dfrac12\Delta\abs{H}^2=\abs{\nabla(JH)}^2+\abs{H}^2.
 \end{eqnarray*}
Then the maximum principle implies that $H\equiv0$, which is a contradiction. Therefore $\Sigma$ is a minimal Legendrian surface in $\mathbb{S}^5$. 
  \end{proof}

\section{Examples of csL Willmore surfaces in \texorpdfstring{$\mathbb{S}^5$}{the unit sphere}}

From the definition we see that complete Willmore Legendrian surfaces, which are minimal surfaces by  \autoref{main thm1} in the last section, are trivial examples of csL Willmore surfaces in $\mathbb{S}^5$. Thus it is very natural and important to find nonminimal csL Willmore surfaces in $\mathbb{S}^5$. This will be done in this section by analyzing a very close relationship between csL Willmore surfaces and csL surfaces in $\mathbb{S}^5$.

Assume that $\Sigma$ is a csL Willmore surface in $\mathbb{S}^5$, then since the variation vector field on $\Sigma$ under Legendrian deformations can be written as $J\nabla u+\frac12u\textbf{R}$ for smooth function $u$ on $\Sigma$ (cf. \cite[Lemma 3.1]{Sm}), we have
\begin{eqnarray*}
0&=&\int_\Sigma\hin{ \overrightarrow{W}(\Sigma)}{J\nabla u+\dfrac12u\textbf{R}}\dif\mu_\Sigma
\\ &=&\int_\Sigma\langle \overrightarrow{W}(\Sigma),J\nabla u\rangle d\Sigma+\int_\Sigma\hin{ \overrightarrow{W}(\Sigma)}{\dfrac12u\textbf{R}}\dif\mu_\Sigma
\\&=&\int_\Sigma \Div\left(J\overrightarrow{W}(\Sigma)-2JH\right)u\dif\mu_\Sigma,
\end{eqnarray*}
where in the last equality we used $\hin{\overrightarrow{W}(\Sigma)}{ \textbf{R}}=-2 \Div(JH)$, by \autoref{pro0}. Therefore $\Sigma$ satisfies the following Euler-Lagrange equation:
\begin{eqnarray}\label{equ6}
\Div\left(J\overrightarrow{W}(\Sigma)-2JH\right)=0.
\end{eqnarray}
\begin{rem}
Note that the coefficient of the Euler-Lagrange equation \eqref{equ6}  for csL Willmore surfaces in $\mathbb{S}^5$ is slightly different with  \cite[equation (1.7)]{Luo}. That is because here we use the notation $H=\trace B$, whereas in \cite{Luo} we defined $H=\frac12\trace B$.
\end{rem}
Then by \eqref{equ4}, $\Sigma$ satisfies the following equation.
\begin{eqnarray*}
\Div\left(\nabla\Div(JH)+JB(JH,JH)-\dfrac12\abs{H}^2JH-4JH\right)=0.
\end{eqnarray*}
In addition, by the four-symmetric of $\left(\sigma_{ijk,l}\right)$ (see \eqref{four-symm.}), a direct computation shows
\begin{align*}
    \Div(JB(JH,JH))=2\trace\hin{ B(\cdot,\nabla_\cdot(JH))}{H}+\dfrac12\nabla_{JH}\abs{H}^2.
\end{align*}
Therefore $\Sigma$ satisfies the following equation
\begin{eqnarray*}
\Delta\Div(JH)+2\trace\hin{ B(\cdot,\nabla_\cdot(JH))}{H}-\dfrac12\abs{H}^2\Div(JH)-4\Div(JH)=0.
\end{eqnarray*}
Therefore we have
\begin{prop}\label{pro1}
Assume that $\Sigma$ is a csL surface in $\mathbb{S}^5$ and $\trace\hin{ B(\cdot,\nabla_\cdot(JH))}{H}=0$, then $\Sigma$ is a csL Willmore surface.
\end{prop}

With the aid of \autoref{pro1}, we can find the following examples of csL Willmore surfaces from csL surfaces in $\mathbb{S}^5$. Firstly, according to \autoref{pro1},  all closed Legendrian surfaces with parallel tangent vector field $JH$, which are exactly minimal surfaces or the Calabi tori (cf. \cite[Proposition 3.2]{LS}),   are csL Willmore surfaces. For reader's convenience, we give some detailed computations as follows.

\begin{eg}[Calabi tori]
For every four nonzero real numbers $r_1, r_2, r_3, r_4$ with $r_1^2+r_2^2=r_3^2+r_4^2=1$, the Calabi torus $\Sigma$ is a csL surface in $\mathbb{S}^5$ defined as follows. 
\begin{align*}
F:&\mathbb{S}^1\times\mathbb{S}^{1} \mapsto \mathbb{S}^5,\\
(t,s)\mapsto&\left(r_1r_3\exp\left(\sqrt{-1}\left(\dfrac{r_2}{r_1}t
+\dfrac{r_4}{r_3}s\right)\right),r_1r_4\exp\left(\sqrt{-1}\left(\dfrac{r_2}{r_1}t-\dfrac{r_3}{r_4}s\right)\right),
r_2\exp\left(-\sqrt{-1}\dfrac{r_1}{r_2}t\right)\right).
\end{align*}
Denote
\begin{align*}
\phi_1=\exp\left(\sqrt{-1}\left(\dfrac{r_2}{r_1}t+\dfrac{r_4}{r_3}s\right)\right),\quad\phi_2
=\exp\left(\sqrt{-1}\left(\dfrac{r_2}{r_1}t-\dfrac{r_3}{r_4}s\right)\right),\quad\phi_3=\exp\left(-\sqrt{-1}\dfrac{r_1}{r_2}t\right),
\end{align*}
then $F(t,s)=\left(r_1r_3\phi_1, r_1r_4\phi_2, r_2\phi_3\right)$. Since
\begin{align*}
\dfrac{\partial F}{\partial t}=\left(\sqrt{-1}r_2r_3\phi_1, \sqrt{-1}r_2r_4\phi_2, -\sqrt{-1}r_1\phi_3\right),\\
\dfrac{\partial F}{\partial s}=\left(\sqrt{-1}r_1r_4\phi_1, -\sqrt{-1}r_1r_3\phi_2, 0\right),
\end{align*}
the induced metric in $\Sigma$ is given by
\begin{align*}
g=dt^2+r_1^2ds^2.
\end{align*}
Let $E_1=\frac{\partial F}{\partial t}, E_2=\frac{1}{r_1}\frac{\partial F}{\partial s}$, then $\{E_1, E_2, \nu_1=\sqrt{-1}E_1, \nu_2=\sqrt{-1}E_2, \textbf{R}=-\sqrt{-1}F\}$ is a local orthonormal frame of $\mathbb{S}^5$ such that $\{E_1,E_2\}
$ is a local orthonormal tangent frame and $\textbf{R}$ is the Reeb field. A direct calculation yields\begin{align*}
\dfrac{\partial\nu_1}{\partial t}=&\left(-\sqrt{-1}\dfrac{r_2^2r_3}{r_1}\phi_1, -\sqrt{-1}\dfrac{r_2^2r_4}{r_1}\phi_2, -\sqrt{-1}\dfrac{r_1^2}{r_2}\phi_3\right),\\
\dfrac{\partial\nu_1}{\partial s}=&\left(-\sqrt{-1}\dfrac{r_2r_3^2}{r_4}\phi_1, \sqrt{-1}\dfrac{r_2r_4^2}{r_3}\phi_2,0\right),\\
\dfrac{\partial\nu_2}{\partial t}=&\left(-\sqrt{-1}\dfrac{r_2r_4}{r_1}\phi_1, \sqrt{-1}\dfrac{r_2r_3}{r_1}\phi_2, 0\right),\\
\dfrac{\partial\nu_2}{\partial s}=&\left(-\sqrt{-1}\dfrac{r_4^2}{r_3}\phi_1,-\sqrt{-1}\dfrac{r_3^2}{r_4}\phi_2, 0\right),\\
\dfrac{\partial\textbf{R}}{\partial t}=&\left(r_2r_3\phi_1, r_2r_4\phi_2, -r_1\phi_3\right),\\
\dfrac{\partial\textbf{R}}{\partial s}=&\left(r_1r_4\phi_1, -r_1r_3\phi_2,0\right).
\end{align*}
Hence,
\begin{align*}
A^{\nu_1}=&-\Re\langle dF,d\nu_1\rangle=\left(\dfrac{r_2}{r_1}-\dfrac{r_1}{r_2}\right)dt^2+r_1r_2ds^2,\\
A^{\nu_2}=&-\Re\langle dF,d\nu_2\rangle=2r_2dtds+r_1\left(\dfrac{r_4}{r_3}-\dfrac{r_3}{r_4}\right)ds^2,\\
A^{\textbf{R}}=&0.
\end{align*}

Thus
\begin{align*}
H=&\left(\dfrac{2r_2}{r_1}-\dfrac{r_1}{r_2}\right)\nu_1+\dfrac{1}{r_1}\left(\dfrac{r_4}{r_3}-\dfrac{r_3}{r_4}\right)\nu_2.
\end{align*}
Moreover $E_1$ and $E_2$ are two parallel tangent vector field. It is obvious that $\Sigma$ is a csL Willmore surface.
\end{eg}

Secondly, we give some examples that $JH$ is not parallel. Mironov \cite{Mir} constructed the following new csL surfaces in $\mathbb{S}^5$. We will verify that Mironov's  examples are in fact csL Willmore surfaces.   

\begin{eg} [Mironov's examples \cite{Mir}] Let $F:\Sigma^2\mapsto\mathbb{S}^5$ be an immersion. Then $F$ is a Legendrian immersion iff
\begin{align*}
\langle F_x,F\rangle=\langle F_y,F\rangle=0.
\end{align*}
Here $\{x,y\}$ is a local coordinates of $\Sigma$ and $\langle,\rangle$ stands for the hermitian inner product in $\mathbb{C}^3$. Set
\begin{align*}
    G=\begin{pmatrix}F\\
    F_x\\
    F_y
    \end{pmatrix},
\end{align*}
then
\begin{align*}
    G\bar G^T=\begin{pmatrix}1&0&0\\
    0&\langle F_x, F_x\rangle&\langle F_x,F_y\rangle\\
    0&\langle F_y,F_x\rangle&\langle F_y,
    F_y\rangle
    \end{pmatrix}:=
    \begin{pmatrix}1&0\\
    0&g
    \end{pmatrix},
\end{align*}
where $g$ is a real positive matrix which is the induce metric of  $\Sigma$. There is a hermitian matrix $\Theta$ such that
\begin{align*}
    G=\begin{pmatrix}1&0\\
    0&g^{1/2}
    \end{pmatrix}e^{\sqrt{-1}\Theta}.
\end{align*}
We compute
\begin{align*}
G\bar G^{T}_x=&\begin{pmatrix}0&-\langle F_x,F_{x}\rangle&-\langle F_x,
F_{y}\rangle\\
\langle F_x,F_x\rangle&\langle F_x,F_{xx}\rangle&\langle F_x,F_{yx}\rangle\\
\langle F_y,F_x\rangle&\langle F_y,F_{xx}\rangle&\langle F_y,F_{yx}\rangle\\
\end{pmatrix}\\
&=\begin{pmatrix}1&0\\
0&\sqrt{g}
\end{pmatrix}e^{\sqrt{-1}\Theta}\left(e^{-\sqrt{-1}\Theta}\right)_x\begin{pmatrix}1&0\\
0&\sqrt{g}
\end{pmatrix}+\begin{pmatrix}1&0\\
0&\sqrt{g}
\end{pmatrix}\begin{pmatrix}0&0\\
0&\left(\sqrt{g}\right)_x
\end{pmatrix}.
\end{align*}
Hence
\begin{align*}
\Re\left(\sqrt{-1}G\bar G^{T}_x\right)=&\Re\sqrt{-1}\begin{pmatrix}0&0&0\\
0&\langle F_x,F_{xx}\rangle&\langle F_x,F_{yx}\rangle\\
0&\langle F_y,F_{xx}\rangle&\langle F_y,F_{yx}\rangle\\
\end{pmatrix}\\
=&\begin{pmatrix}0&0\\
0&A^{\sqrt{-1}F_x}
\end{pmatrix}\\
&=\begin{pmatrix}1&0\\
0&\sqrt{g}
\end{pmatrix}\Re\left(\sqrt{-1}e^{\sqrt{-1}\Theta}\left(e^{-\sqrt{-1}\Theta}\right)_x\right)\begin{pmatrix}1&0\\
0&\sqrt{g}
\end{pmatrix},
\end{align*}
which implies
\begin{align*}
\begin{pmatrix}0&0\\
0&g^{-1/2}A^{\sqrt{-1}F_x}g^{1/2}
\end{pmatrix}=\Re\left(\sqrt{-1}e^{\sqrt{-1}\Theta}\left(e^{-\sqrt{-1}\Theta}\right)_x\right).
\end{align*}
Similarly,
\begin{align*}
\begin{pmatrix}0&0\\
0&g^{-1/2}A^{\sqrt{-1}F_y}g^{1/2}
\end{pmatrix}=\Re\left(\sqrt{-1}e^{\sqrt{-1}\Theta}\left(e^{-\sqrt{-1}\Theta}\right)_y\right).
\end{align*}
The Lagrangian angle is then given by $\theta=tr\Re\Theta$. The above discussion implies that
\begin{align*}
    J\nabla\theta=H.
\end{align*}

Let $a, b, c$ are three positive constants and consider the following immersion
\begin{align*}
F: &\mathbb{S}^1\times\mathbb{S}^1\mapsto\mathbb{S}^5,
\\(x,y)&\mapsto\left(\phi(x) e^{\sqrt{-1}ay}, \psi(x) e^{\sqrt{-1}by}, \zeta(x) e^{-\sqrt{-1}cy}\right),
\end{align*}
where
\begin{align*}
&\phi(x)=\sqrt{\dfrac{c}{a+c}}\sin x,
\\& \psi(x)=\sqrt{\dfrac{c}{b+c}}\cos x,\\
&\zeta(x)=\sqrt{\dfrac{a\sin^2x}{a+c}+\dfrac{b\cos^2x}{b+c}}=\sqrt{\dfrac{ab+u(x)}{(a+c)(b+c)}},
\end{align*}
where
\begin{align*}
u(x)=\dfrac{c\left(a+b+(b-a)\cos(2x)\right)}{2}.
\end{align*}
One can check that $F$ is a Legendrian immersion. Denote $\Sigma:=F\left(\mathbb{S}^1\times\mathbb{S}^1\right)$.
Notice that
\begin{align*}
F_x=&\left(\sqrt{\dfrac{c}{a+c}}\cos x e^{\sqrt{-1}ay}, -\sqrt{\dfrac{c}{b+c}}\sin x e^{\sqrt{-1}by}, \dfrac{-c(b-a)\sin(2x)}{2\sqrt{(a+c)(b+c)(ab+u(x))}}e^{-\sqrt{-1}cy}\right),\\
F_y=&\left(\sqrt{-1}a\phi(x)e^{\sqrt{-1}ay},\sqrt{-1}b\psi(x) e^{\sqrt{-1}by}, -\sqrt{-1}c\zeta(x)e^{-\sqrt{-1}cy}\right).
\end{align*} The induced metric $g$ is given by
\begin{align*}
g=&\left[\dfrac{c\cos^2x}{a+c}+\dfrac{c\sin^2x}{b+c}+\dfrac{c^2(b-a)^2\sin^2(2x)}{4(a+c)(b+c)(ab+u(x))}\right]dx^2\\
&+\left[a^2\times\dfrac{c\sin^2x}{a+c}+b^2\times\dfrac{c\cos^2x}{b+c}+c^2\left(\dfrac{a\sin^2x}{a+c}+\dfrac{b\cos^2x}{b+c}\right)\right]dy^2\\
=&\dfrac{u(x)}{ab+u(x)}dx^2+u(x)dy^2\\
:=& e^{2p(x)}dx^2+e^{2q(x)}dy^2.
\end{align*}
A strait forward calculation yields that
\begin{align*}
A^{\sqrt{-1}F_x}=\Re\begin{pmatrix}0&\sqrt{-1}\langle F_x,F_{xy}\rangle\\
-\sqrt{-1}\langle F_{xy},F_x\rangle&0\end{pmatrix}=\begin{pmatrix}0&c\left(1-e^{2p(x)}\right)\\
c\left(1-e^{2p(x)}\right)&0\end{pmatrix},\\
A^{\sqrt{-1}F_y}=\Re\begin{pmatrix}\sqrt{-1}\langle F_x,F_{yx}\rangle&0\\
0&\sqrt{-1}\langle F_y,F_{yy}\rangle\end{pmatrix}=\begin{pmatrix}c\left(1-e^{2p(x)}\right)&0\\
0&(a+b-c)e^{2q(x)}-abc\end{pmatrix}.
\end{align*}
We get
\begin{align*}
    \Re\left(\sqrt{-1}e^{\sqrt{-1}\Theta}\left(e^{-\sqrt{-1}\Theta}\right)_x\right)=\begin{pmatrix}0&0&0\\
    0&0&\dfrac{abc}{u\sqrt{ab+u}}\\
    0&\dfrac{abc}{u\sqrt{ab+u}}&0
    \end{pmatrix},\\
    \Re\left(\sqrt{-1}e^{\sqrt{-1}\Theta}\left(e^{-\sqrt{-1}\Theta}\right)_y\right)=\begin{pmatrix}0&0&0\\
    0&\dfrac{abc}{u}&0\\
    0&0&(a+b-c)-\dfrac{abc}{u}
    \end{pmatrix}.
\end{align*}
Thus
\begin{align*}
H^{\sqrt{-1}F_x}=0,\quad H^{\sqrt{-1}F_y}=a+b-c.
\end{align*}
We get
\begin{align*}
H=\dfrac{a+b-c}{u(x)}\sqrt{-1}\dfrac{\partial}{\partial y},
\end{align*}
and
\begin{align*}
\nabla_{\partial_x}\left(\sqrt{-1}H\right)=\dfrac{(a+b-c)u_x}{2u^2}\dfrac{\partial}{\partial y},\quad\nabla_{\partial_y}\left(\sqrt{-1}H\right)=\dfrac{(ab+u)(a+b-c)u_x}{2u^2}\dfrac{\partial}{\partial x}.
\end{align*}
In particular
\begin{align*}
\Div\left(\sqrt{-1}H\right)=0.
\end{align*}
Hence $\Sigma$ is csL. Moreover
\begin{align*}
\sum_{i=1}^2\hin{ B\left(e_i,\nabla_{e_i}\left(JH\right)\right)}{H}=0.
\end{align*}
Therefore, $\Sigma$ is a csL Willmore surface in $\mathbb{S}^5$.
\end{eg}


\end{document}